\documentclass[10pt]{amsart}
\usepackage[margin=1.03in]{geometry}
\usepackage{amsmath,amsfonts,amssymb,amsthm,epsfig,color,tikz,hyperref,tkz-euclide}
\usepackage{fancyvrb}
\usepackage{MnSymbol} 
\usepackage{float}
\usepackage{graphicx}

\newcommand{\bb}{\mathbb}
\newcommand{\C}{\bb C}

\newcommand{\Z}{\bb Z}
\newcommand{\R}{\bb R}
\newcommand{\RR}{\R}
\newcommand{\bR}{\R}

\newcommand{\Q}{\bb Q}

\newcommand{\T}{\mathbb T}
\newcommand{\bT}{\mathbb T}
\newcommand{\bC}{\mathbb C}
\newcommand{\bZ}{\mathbb Z}

\newcommand{\cI}{\mathcal{I}}

\newcommand{\vol}{\operatorname{vol}}
\newcommand{\covol}{\operatorname{covol}}

\newcommand{\CC}{\mathbb{C}}
\newcommand{\ZZ}{\mathbb{Z}}

\newcommand{\twovector}[2]{\begin{pmatrix} #1 \\ #2 \end{pmatrix}}

\newtheorem{Theorem}{Theorem}
\newtheorem{Cor}[Theorem]{Corollary}
\newtheorem{Prop}[Theorem]{Proposition}
\newtheorem{lemma}[Theorem]{Lemma}
\newtheorem*{lemma*}{Lemma}
\newtheorem*{theorem*}{Theorem}

\theoremstyle{definition} % No italics in this environment
\newtheorem{Def}{Definition}

\newtheorem{rem}{Remark}

\numberwithin{equation}{section}
\numberwithin{Def}{section}
\numberwithin{Theorem}{section}
\begin{document}

\title{Counting Tripods on the Torus}
\author{Jayadev S. Athreya}
\author{David Aulicino}
\author{Harry Richman}
\address{Department of Mathematics, University of Washington, Padelford Hall, Seattle, WA 98195, USA}
\email{jathreya@uw.edu}
\address{Department of Mathematics, Brooklyn College,
Room 1156, Ingersoll Hall 
2900 Bedford Avenue 
Brooklyn, NY 11210-2889, USA}
\address{Department of Mathematics, The Graduate Center, CUNY
365 Fifth Avenue
New York, NY 10016
USA}
\email{david.aulicino@brooklyn.cuny.edu}
\address{Department of Mathematics, University of Washington, Padelford Hall, Seattle, WA 98195, USA}
\email{hrichman@uw.edu}

\date{}
  \thanks{J.S.A. was partially supported by NSF CAREER grant DMS 1559860 and NSF DMS 2003528}
    \thanks{D.A. was partially supported by NSF DMS - 1738381 and several PSC-CUNY grants}
\begin{abstract} Motivated by the problem of counting finite BPS webs, we count certain immersed metric graphs, \textit{tripods}, on the flat torus. Classical Euclidean geometry turns this into a lattice point counting problem in $\mathbb C^2$, and we give an asymptotic counting result using lattice point counting techniques.
\end{abstract}

\maketitle

\section{Introduction}\label{sec:intro}

Given two points $z, w$ in the plane $\C$, the set of points $p$ that minimize the sum of distances $|z-p| + |w-p|$ is exactly the line segment connecting $z$ and $w$.  At any point $p$ on this segment, the line segments between $z$ and $p$ and $w$ and $p$ have angle $2\pi/2 = \pi$. Given \emph{three points} $z, w, u$ in $\C$ such that the triangle they form has largest angle at most $2\pi/3$,  the \emph{Fermat point} $p$ minimizes the sum of the lengths $|z-p| + |w-p| + |u-p|$. A classical result in Euclidean geometry says that the angles between the line segments connecting $p$ to $z, w, u$ are $2\pi/3$. We call the configuration of line segments a \emph{tripod}.

An integral lattice point $m+ in \in \Z[i]$ is called \emph{primitive} if $m$ and $n$ are coprime.  A classical result states that as $R$ tends to infinity, the number of primitive points in the circle of radius $R$ centered at the origin grows like $\frac{R^2}{\zeta(2)}$. These points correspond to embedded closed geodesics in the flat torus $\C/\Z[i]$. Inspired by the work of \cite{GMN}, we ask for a similar asymptotic for the number of \emph{immersed} graphs in flat tori that are the projections of tripods whose vertices are at points of a unimodular lattice. 

Let $\Lambda$ be a lattice in $\bC$.  Throughout, we will assume that $\Lambda$ has unit covolume.  A primitive vector in the complex plane descends to a primitive closed trajectory on $\bT = \bC/\Lambda$ with the origin marked.  Analogously, a tripod on the torus will be an immersed tripod from the plane such that the Fermat point descends to a point $p \not= 0$ and all three other endpoints of the line segments descend to $0$.  More formally, $\cI$ is an  isometrically \emph{immersed} copy in $\T$ of a metric graph $\mathcal G  = \mathcal G(\ell_1, \ell_2, \ell_3) \subset \C$ (see Figure~\ref{fig:tripodC}) given by positive parameters $\ell_1, \ell_2, \ell_3$, $$\mathcal G = \{ t: 0 \le t \le \ell_1\} \bigcup \{ t e^{2\pi i/3}: 0 \le t \le \ell_2 \} \bigcup \{ t e^{4\pi i/3}: 0 \le t \le \ell_3 \}.$$

The image of the tripod is in fact an immersed copy of an equiangular $\Theta$-graph; that is, a graph with two vertices with three edges between them. The vertices are the point $0$ and the tripod point $p$, and it is not difficult to check that the line segments must meet with angle $2\pi/3$ at $0$. We will later discuss (see Section~\ref{sec:triprop}) that associated to each tripod is a (minimal) cover of the torus where the tripod becomes embedded.  The degree of the cover is the number of transverse self-intersections of the original tripod, and the embedded tripod gives a representation of the covering torus as an equiangular hexagon with parallel sides identified by translation.

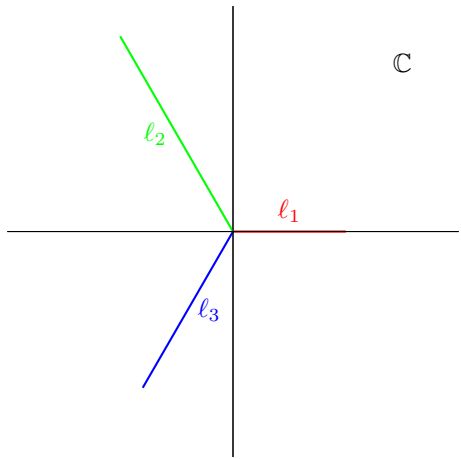
\begin{figure}[h]
\begin{tikzpicture}[scale=1.5]

\draw(-2,0)--(2,0);
\draw(0, -2)--(0,2);

\draw[thick, red](0,0)--(1, 0)node[midway, above]{$\ell_1$};
\draw[thick, green](0,0)--(120:2)node[midway, left]{$\ell_2$};
\draw[thick, blue](0,0)--(240:1.6)node[midway, right]{$\ell_3$};

\draw(-2,0)--(2,0);
\draw(0, -2)--(0,2);

\node at (1.5,1.5){$\C$};

\end{tikzpicture}
\caption{The graph $\mathcal G(\ell_1, \ell_2, \ell_3) \subset \C $.}
\label{fig:tripodC}
\end{figure}

\begin{Def}\label{def:tripod} 
A \textit{tripod} $\downY$ consists of a pair $(\mathcal G(\ell_1, \ell_2, \ell_3), \mathcal I)$, where $\cI: \mathcal G \rightarrow \T$ is an isometric immersion, with $\mathcal I(0) = p$, and $$\mathcal I(\ell_1) = \mathcal I(\ell_2 e^{2\pi i/3} ) =  \mathcal I (\ell_3 e^{4\pi i/3}) = 0.$$ The \emph{length} of the tripod is denoted $\ell(\downY)$, and is given by $$\ell(\downY) = \ell_1 + \ell_2 + \ell_3.$$ We say a tripod $\downY$ is \emph{primitive} if it is not a scaled copy of another tripod.  A primitive tripod is called \emph{reduced}  if the only lattice points lying on its legs are at its endpoints and the tripod point is not at a lattice point.
\end{Def}

The distinction between primitive and reduced lattices is important for future work.  In the case of lattice points, the concepts are identical: a line segment from the origin to a lattice point passes through another lattice point if and only if it is a scaled copy of the vector from the origin to that other lattice point.  However, this is not the case for tripods.  In future work, we plan to analyze tripods on higher genus translation surfaces, which admit cone points.  In the presence of cone points, it is natural to define a tripod as having endpoints at the cone points and no cone points on either of the legs or at the tripod point.  This is essential because the angle at the cone point of a translation surface is only well-defined modulo $2\pi$.

\begin{figure}

\begin{tikzpicture}[scale=2.5]
\draw(0,0)--(0, 1)--(1,1)--(1,0)--cycle;

\draw[thick, blue](0,0)--(0.349, 0.2549);
\draw[thick, red](0.349, 0.2549)--(1,0);
\draw[thick, green](0.349, 0.2549)--(0.17449, 1);
\draw[thick, green](0, 1)--(0.17449, 0);

\end{tikzpicture}
\caption{A (primitive) tripod $\downY$ drawn in a fundamental domain for $\T$.}
\label{fig:tripodexample}
\end{figure}
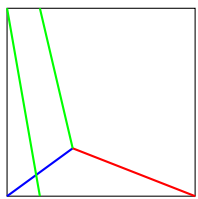

Given a lattice $\Lambda$, let $$N_{\downY} (R, \Lambda) := \# \{ \downY :  \downY \mbox { primitive with endpoints in } \Lambda \mbox { and } \ell(\downY) \le R\}.$$ 
Furthermore, let $$N_{red, \downY} (R, \Lambda) := \# \{ \downY :  \downY \mbox { reduced with endpoints in } \Lambda \mbox { and } \ell(\downY) \le R\}.$$ 
What is the asymptotic behavior of $N_{\downY}(R, \Lambda)$? Our main result is:

\begin{Theorem}\label{theorem:main}
For all (unit covolume) lattices $\Lambda$ in $\bC$,

%$$\lim_{R \rightarrow \infty} \frac{N_{\downY} (R, \Lambda) \covol(\Lambda)^2}{R^4} = \frac{1}{\zeta(4)} \cdot \frac{\sqrt 3 \pi}{ 24 } = \frac{15 \sqrt 3}{4 \pi^3}.$$
$$\lim_{R \rightarrow \infty} \frac{N_{\downY} (R, \Lambda)}{R^4} = \frac{1}{\zeta(4)} \cdot \frac{\sqrt 3 \pi}{ 24 } = \frac{15 \sqrt 3}{4 \pi^3}.$$

\end{Theorem}

\noindent We will prove this claim by turning our problem into a problem of counting pairs $(z, w) \subset \Lambda = \Z+\Z\tau$ satisfying certain conditions. The $1/\zeta(4)$ term in our theorem arises from the fact that we will be counting pairs 
$$z = a+ b \tau,\quad w = c+d\tau ,\qquad  a, b, c, d \in \Z$$ 
with $\gcd(a, b, c, d) = 1.$ The term $\frac{\sqrt 3 \pi}{24}$ represents the volume of a region in $\C^2$ in which we will be counting dilations of sets of points.

To clarify how the concepts of primitive versus reduced tripods affect their asymptotics, we prove the following two theorems.  Let $N_{nonred, \downY} (R, \Lambda)$ be the count of the non-reduced tripods up to length $R$.

\begin{Theorem}\label{theorem:AEreduced}
For almost every lattice $\Lambda$,
$$N_{nonred, \downY} (R, \Lambda) = o(R^4).$$
\end{Theorem}

On the other hand, we show that Theorem~\ref{theorem:AEreduced} does not hold for all lattices.

\begin{Theorem}\label{theorem:MaximalReduced}
Let $\zeta = e^{2\pi i/6}$.  Then 
$$N_{nonred, \downY} (R, \ZZ + \ZZ\zeta) \geq C R^4$$
for some positive constant $C$, for sufficiently large $R$.
\end{Theorem}

\subsection{Lifting}\label{sec:lifting} Note that if we lift a (primitive) tripod $\downY$ from $\T$ to $\C$, we obtain a center point $\tilde{p}$ and segments emanating from $\tilde{p}$ to points in $\Lambda$. We can always choose our lift so that one of these points is $0$, and we call the other two $z$, $w$ with, say $\arg(z) < \arg(w)$. 
To remove ambiguity, we insist that 
the point $\tilde p$ lies in the sector of $\C$ specified by
$$  0 \leq \arg(\tilde p) < 2\pi /3.$$
%the largest angle of the triangle $\Delta(0, z, w)$ is at $0$, that is, $$|w-z| > |z|, |w-z| > |w|.$$ 
Only certain pairs $(z, w)$ will yield triangles which have inscribed tripods. The length of the tripod can be computed explicitly in terms of $z$ and $w$, thus turning our problem into a lattice point counting problem in $\bC^2$.

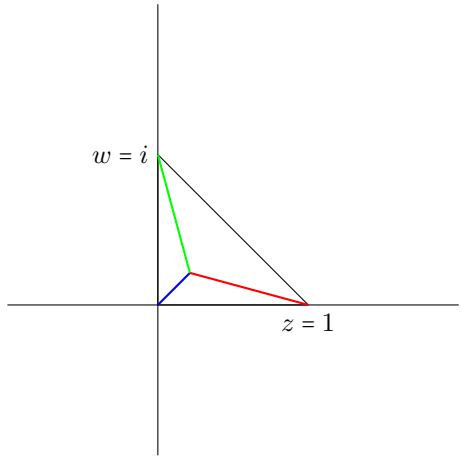
\begin{figure}
\begin{tikzpicture}[scale = 2.0]

\draw(-1,0)--(2,0);
\draw(0, -1)--(0,2);
\draw(0,0)--(0,1)node[left]{$w=i$};

\draw(0,0)--(1,0)node[below]{$z=1$};

\draw(1,0)--(0,1);

\draw[thick, blue](0,0)--(0.21325, 0.21325);
\draw[thick, green](0,1)--(0.21325, 0.21325);
\draw[thick, red](1,0)--(0.21325, 0.21325);
\end{tikzpicture}

\caption{Tripod $\downY$ inscribed in triangle $\Delta(0, 1, i)$. The tripod point is $\left( \frac{1}{3+\sqrt 3}, \frac{1}{3 + \sqrt 3}\right)$.}
\label{fig:tripod1i}
\end{figure}

%\subsection{Webs} Finite BPS webs ~\cite{GMN}  can be thought of as certain classes of immersed metric graphs in Riemann surfaces with flat geometries given by lists of holomorphic $k$-differentials. This work is motivated by the general problem of counting finite BPS webs of bounded length.   \note{Say something substantial here about how it connects to BPS webs.}

%
\subsection{Differentials, saddle connections, and tripods}\label{sec:differentials} Our problem is also inspired by the problem of counting \emph{saddle connections} for \emph{quadratic differentials}. Given a holomorphic quadratic differential $q$ on a compact Riemann surface $X$, there is a singular flat metric associated to $q$ with conical singularities of angle $(k+2)\pi$ at zeros of order $k$ of the differential. A \emph{saddle connection} is a geodesic trajectory connecting two singular points with no singularities in its interior. Alternatively, one can think of it as a regular point $p$ on $(X, q)$ with two geodesic rays $\gamma_1$, $\gamma_2$ emanating from $p$ with an angle of $\pi$, each terminating in a singular point. Note that angles of $\pi$ correspond to different local choices of a square root of $q$. Of course there are (infinitely) many choices of $p$ in this setting. The problem of counting saddle connections is very well studied, see~\cite{EskinSurvey}, for example. Our problem can be thought of as a problem naturally adapted to the setting of holomorphic \emph{cubic} differentials $c$, where it is natural to consider trajectories emanating from a point at angles of $2\pi/3$, each corresponding to a different choice of cube root of the differential $c$. More generally, finite BPS webs arise in the work of \cite{GMN} and are associated to lists of differentials of varying orders.  This work is motivated by the general problem of counting finite BPS webs of bounded length.

%\begin{rem}
%For primitive vectors on flat tori, the independence of the $1/\zeta(2)$ proportion on a particular lattice is significant in that it is the simplest instance of the $\splin$-invariance of Siegel-Veech constants when generalizing to higher genus quadratic differentials as mentioned above.  Indeed the independence of the constant on the lattice in the tripod asymptotic growth rate is particularly intriguing in this context because there is no obvious group action.  
%In particular, the action by $\splin$ on the space of flat tori does not preserve the tripod because it does not preserve angles.
%
%\note{QUESTION FOR US: WHAT LEADS TO THE INVARIANCE?  WHAT OTHER SETS ARE INVARIANT?  IS IT COMMON THAT SETS ARE INDEPENDENT OF THE LATTICE?  WE SHOULD HAVE SOME IDEA HERE FOR HOW SPECIAL THIS IS OR NOT BEFORE WRITING THIS REMARK.}
%\end{rem}

\begin{rem}
We remark that recent work of Koziarz-Nguyen~\cite{KoziarzNguyen} shows that the leading term for the normalized asympotics for counting certain types of triangulations on surfaces is in $\Q \cdot (\sqrt 3\pi)^{N}$ for an appropriate power $N$.  Nevertheless, we do not see an obvious relation between the results.  There is no primitivity assumption in the work of \cite{KoziarzNguyen}, which accounts for the $\zeta(4)$ factor in our work.  However, if the primitivity is removed, then the $R^4$ growth rate of $N_{\downY} (R, \Lambda)$ behaves as $\sqrt{3}\pi/24 \not \in \Q \cdot (\sqrt{3}\pi)^4$.  So there does appear to be a fundamental difference in the objects being counted.
\end{rem}

\subsection*{Acknowledgments} We would like to thank Vincent Delecroix and Andy Neitzke for useful discussions; morally, they are coauthors of this paper.  We thank Marty Lewinter for showing us the connection between tripods and Fermat points.  We would like to thank the Mathematical Sciences Research Institute (MSRI), where this project started in Fall 2019 at the semester program on \textit{Holomorphic Differentials in Mathematics and Physics}.

\subsection*{Organization} 
In \S\ref{sec:fermat}, we explain the Euclidean geometry which allows us to translate the problem to a lattice point counting problem. In \S\ref{sec:triprop}, we summarize some nice properties of tripods, lengths, intersections, and covers. In \S\ref{sec:lattice}, we state precisely the lattice point counting problem in $\C^2$ and the lattice point counting results which we use to solve the problem. In \S\ref{sec:volumes}, we compute the volume of a region in $\C^2$ which gives us the main term in the asymptotic formula. In \S\ref{sec:NonRedTripods}, we prove Theorems~\ref{theorem:AEreduced} and~\ref{theorem:MaximalReduced} concerning non-reduced tripods.

\section{Fermat points and Steiner trees}
\label{sec:fermat} 
\subsection{Inscribed Tripods}
\label{sec:inscribed} 
We recall some beautiful facts from classical Euclidean geometry which are crucial for our translation of our counting problem to a lattice point counting problem. The following is due to Toricelli. The problem was posed to him by Fermat, and published by Toricelli's student Viviani. An excellent history of this problem (and the more general Steiner tree problem) can be found here~\cite{Brazil}.

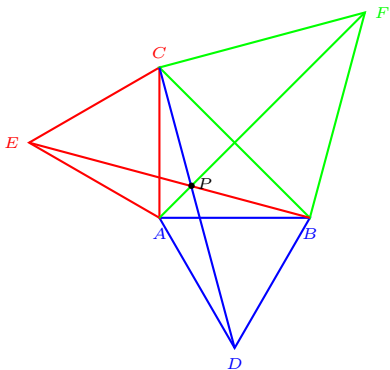
\begin{figure}[h]
\begin{tikzpicture}[scale=2.0]

\draw[thick, blue](0,0)node[below]{\tiny $A$}--(1,0)node[below]{\tiny $B$}--(300:1)node[below]{\tiny $D$}--cycle;
\draw[thick, red](0,0)--(0,1)node[above]{\tiny $C$}--(150:1)node[left]{\tiny $E$}--cycle;
\draw[thick, green](1,0)--(0,1)--(1.366, 1.366)node[right]{\tiny $F$}--cycle;
\draw[thick, blue](300:1)--(0,1);
\draw[thick, green](0,0)--(1.366, 1.366);
\draw[thick, red](150:1)--(1,0);

\filldraw (0.21325, 0.21325)circle[radius=0.5pt];

\node[right] at (0.19, 0.23){\tiny $P$};

%\draw(0,0)node[below]{$A$}--(30:1)node[right]{$B$}--(100:2)node[above]{$C$}--cycle;
%
%\draw[thick, blue](30:1)--(100:2);
%
%\draw[thick, green](0,0)--(30:1)--(330:1)--cycle;
%
%\draw[thick, red](0,0)--(100:2)--(160:2)--cycle;
%
%\draw[thick, red][name path=A--B](160:2)--(30:1);
%
%\draw[thick, green][name path=C--D](100:2)--(330:1);
%
%\path [name intersections={of=A--B and C--D,by=E}];
%
%\draw[thick, blue](0,0)--(E);
\end{tikzpicture}
\caption{Constructing a tripod (aka Steiner tree) $\downY$ inscribed in a triangle $ \Delta(ABC) $.}
\label{fig:tripodconstruction}
\end{figure}

\begin{Theorem}
\label{thm:tripod-triangle} 
A triangle $\Delta(ABC)$ contains an inscribed tripod if and only if the largest angle is at most $2\pi/3$. In this case, the tripod $\downY(ABC)$  is constructed by constructing equilateral triangles (which do not intersect the original triangle) on each side, and drawing lines connecting the opposite vertex of this equilateral triangle to the opposite vertex of the original triangle.  
That is, 
\begin{itemize}
\item 
build an equilateral triangle $\Delta(ABD)$ on the side $AB$, 
\item
an equilateral triangle $\Delta(ACE)$ on the side $AC$, 
\item
an equilateral triangle $\Delta(BCF)$ on the side $BC$, 
\item 
and build the segments $DC$,  $BE$, and $AF$. 
\end{itemize}
The three line segments ($DC, BE, AF)$  intersect at a common point, $P$, and the line segments $AP$, $BP$, $CP$, form the tripod. 
See Figure~\ref{fig:tripodconstruction}. 
Furthermore, this point $P$ is the unique point which minimizes the quantity $$|AQ| + |BQ| + |CQ|$$ over all points $Q$ in the plane. The tripod is known as the \emph{Steiner tree} associated to the points $A, B, C$~\cite{Steinertreebook}. Moreover, the length of the tripod is given by 
\begin{align*}
\ell(\downY(ABC)) &= |AP| + |BP| + |CP| \\
&= |DC| = |BE| = |AF|.
\end{align*}
\end{Theorem}

\noindent Applying the above result to the triangle $\Delta(0, z, w)$ for $$z, w \in \C \mbox{ such that }\arg(z) < \arg(w) 
\mbox{ and }\Delta(0, z, w) \mbox{ has all angles at most }2\pi/3, $$ 
we obtain the following result.

\begin{lemma}
\label{lemma:length} 
%Suppose $z,w$ are complex number with 
Let $\downY(z, w)$ denote the tripod inscribed in the triangle $\Delta(0, z, w)$, and let $p$ denote the tripod point of $\downY(z,w)$. 
Then 
\begin{align}\label{eq:tripodlength}
\ell(\downY(z, w)) &= |e^{i\pi/3}z + e^{-i\pi/3}w|
\end{align}
and
\begin{equation}
\label{eq:tripod-arg}
\arg(p) = \arg(e^{i\pi/3}z + e^{-i\pi/3}w) .
\end{equation}
\end{lemma}

\begin{proof} We write $\downY$ for $\downY(z, w)$. Let $p$ denote the tripod point. Then 
$$\ell(\downY) = |p| + |w-p| + |z-p|.$$ 
On the other hand, we assume that $\arg(z) < \arg(w)$ and by the fact that the line segments $p0$, $pz$, and $pw$ are at angles of $2\pi/3$, we have that the three complex numbers 
$$e^{i \pi/3} (z-p), \quad  p,\quad  e^{-i\pi /3}(w-p)$$ 
%$$e^{i 2\pi/3} (z-p), \quad  - e^{-i2\pi/3} p,\quad  w-p$$ 
are parallel (that is, their ratios are real and positive). Therefore, the magnitude of their sum is the sum of their magnitudes. 
Since 
%$$|z-p| = |e^{i 2\pi/3} (z-p)|,\quad  |p| = |- e^{-i2\pi/3} p|,\quad |w-p| =| w-p|,$$
% and 
$$ % -\left(e^{i2\pi/3} + e^{-i2\pi/3} \right) = 
1 = e^{i\pi/3} + e^{-i\pi/3},$$
it follows that
\begin{align}\label{eq:tripodlengthcomplex}
\ell(\downY) &= |p| + |w-p| + |z-p|\\ \nonumber
&= |p| + |e^{-i \pi/3}(w-p)| + |e^{i \pi/3}(z-p)|\\ \nonumber
&= |p + e^{-i \pi/3}(w-p) + e^{i\pi/3}(z-p)|\\ \nonumber 
&= |e^{-i \pi/3}w + e^{i \pi/3} z| . \nonumber
\end{align}
Our earlier observation that
the complex numbers $p, e^{-i \pi/3}(w-p), e^{i \pi/3}(z-p)$ all have the same argument implies that
\begin{align*}
\arg(p) &= \arg(p + e^{-i \pi/3}(w-p) + e^{i \pi/3}(z-p)) \\
&= \arg(e^{-i \pi/3}w + e^{i \pi/3}z)
\end{align*}
as claimed.
\end{proof}

\begin{rem}
There are also other ways of defining lengths of tripods. For example, given a tripod $\downY$ inscribed in the triangle $\Delta(ABC),$ we could define its \emph{triangle length} $\ell_{\Delta}$ to be $$\ell_{\Delta} (\downY) = |AB| + |AC| + |BC|.$$ 
\end{rem}

\subsection{Steiner trees}\label{sec:steiner} More generally, it could be interesting to consider projections to the torus of solutions to the \emph{Euclidean Steiner tree problem} with integer vertices: given $N$ points in the plane, find the connected embedded graph with minimal total length with vertices at these points. For two points this is of course the straight line, and more generally, it is not hard to see that the minimizer must be a tree. 

\section{Tripod Properties}\label{sec:triprop}

In this section, we consider the number of self-intersections of a tripod on a torus 
and the number of subregions 
that a tripod divides a torus into. We relate these tripod properties to its lengths, defined in the previous section.

%that result on a torus between tripod segments.

\subsection{Lattice index and tripod lengths}
Given a tripod $\downY$ in a lattice $\Lambda$,
we define the {\em lattice index} of $\downY$ in $\Lambda$
as the index $[\Lambda : \Lambda(\downY)]$,
where $\Lambda(\downY)$ is the minimal lattice in $\R^2$
which contains $\downY$ as a tripod.
We define the {\em volume} of a tripod
as the covolume of the lattice $\Lambda(\downY)$. Recall that the tripod length is given by $\ell(\downY) = \ell_1 + \ell_2 + \ell_3$. We define the $L_2$ length of a tripod by
$$L_{2}(\downY) = (\ell_1^2 + \ell_2^2 + \ell_3^2)^{1/2}.$$

%Recall that for a tripod $\downY$ with lengths $\ell_1, \ell_2, \ell_3$
%we define
%\begin{align*}
%\ell(\downY) = \ell_1 + \ell_2 + \ell_3,
%\qquad
%\end{align*}

\begin{Prop}
\hfill 
\begin{enumerate}
\item 
The tripod volume $vol(\downY)$ is related to the tripod lengths by
\begin{equation*}
\vol(\downY) = \frac{\sqrt 3}{4} \left( \ell(\downY)^2 - L_2(\downY)^2 \right) .
\end{equation*}

\item 
The lattice index $n(\downY,\Lambda)$ is related to the tripod lengths
$\ell(\downY)$ and $L_2(\downY)$ by
%\begin{equation*}
%n(\downY, \Lambda) = \frac{\covol(\Lambda(\downY))}{\covol(\Lambda)}
%= \frac{\sqrt 3}{4 \, \covol(\Lambda)} \left( \ell(\downY)^2 - \ell^2(\downY)^2 \right) .
%\end{equation*}
\begin{equation*}
n(\downY, \Lambda) = \covol(\Lambda(\downY))
= \frac{\sqrt 3}{4} \left( \ell(\downY)^2 - L_2(\downY)^2 \right) .
\end{equation*}
\end{enumerate}
\end{Prop}
\begin{proof}
Suppose $\downY$ spans the triangle with vertices $0, z, w$. 
Then the area of a fundamental domain of $\Lambda(\downY)$
is twice the area of the triangle $\Delta(0,z,w)$.
The tripod dissects the triangle $\Delta(0,z,w)$ into three subtriangles,
each with internal angle $2\pi /3$.
By adding up these areas, we have
$$
Area(\Delta(0,z,w)) = \frac{\sqrt{3}}{4} \left( \ell_1\ell_2 + \ell_1\ell_3 + \ell_2 \ell_3 \right).
$$
Therefore,
\begin{align*}
covol(\Lambda(\downY)) 
&= \frac{\sqrt{3}}{2} \left( \ell_1\ell_2 + \ell_1\ell_3 + \ell_2 \ell_3 \right) 
= \frac{\sqrt 3}{4}\left( (\ell_1+\ell_2+\ell_3)^2 - (\ell_1^2 + \ell_2^2 + \ell_3^2) \right) .
\end{align*}
This proves the first part of the proposition. The second part
follows from the first part by the relation
%$n(\downY, \Lambda) = \frac{covol(\Lambda(\downY))}{covol(\Lambda)}$.
$$n(\downY, \Lambda) = \covol(\Lambda(\downY)).$$
\end{proof}

\subsection{Counting intersections and regions}

\begin{Prop}\label{prop:intersections}
Suppose $\downY$ is a tripod on the torus $\T = \C / \Lambda$
with only transverse self-intersections, with index $n$.
Then the number of self-intersections of $\downY$ on $\T$ is $n-1$.
\end{Prop}
\begin{proof}
Suppose the tripod $\downY$ spans the triangle with vertices $0, z, w$.
Let $\Lambda(\downY)$ 
denote the lattice in $\C$ spanned by $z$ and $w$;
we have $\Lambda(\downY) \subset \Lambda$ by assumption that $\downY$
is a tripod in $\Lambda$.

Let $n = [\Lambda : \Lambda(\downY)]$.
If we lift the tripod $\downY$ from $\C / \Lambda$
to $\C / \Lambda(\downY)$, 
then the preimage can be expressed as a union of $n$ tripods 
$\downY_1, \downY_2, \ldots, \downY_n$
which are translates of each other.
We claim that for each $i\neq j$,
the intersection $\downY_i \cap \downY_j$ consists of exactly two points
in the covering torus $\C / \Lambda(\downY)$.
To verify this claim, first observe that 
$\downY_i \cap \downY_j = \downY_i \cap (\downY_j+\epsilon)$
as $\epsilon$ varies over a small neighborhood $U$ of zero in $\RR^2$,
with the neighborhood chosen such that the intersection remains transverse.
By moving the copy $\downY_j$ toward $\downY_i$, along a path that keeps the intersection transverse, 
we have
$\downY_i \cap \downY_j = \downY_i \cap (\downY_i+\epsilon)$.
%More generally
%if we consider $\downY_i \cap (\downY_j+\epsilon)$ 
%as $\epsilon$ varies over a small neighborhood $U$ of zero in $\RR^2$,
%the size of the intersection is constant for those $\epsilon$ such that the intersection is transverse, which holds for an open dense subset of $U$.
%By assumption $\downY_i$ and $\downY_j$ intersect transversely. 
Finally, we verify that 
$\downY_i \cap (\downY_i + \epsilon) = 2$
when intersecting a tripod with a small translate of itself,
as demonstrated in Figure~\ref{fig:tripod-intersection}.
\begin{figure}[h]
\centering
\begin{tikzpicture}
\draw (0,0) -- (-2,0);
\draw (0,0) -- (60:1);
\draw (0,0) -- (-60:2) -- (-30:3.46);
\draw (-60:2) -- (0.5,-2.60);
\draw[dashed] (0,-0.2)+(0,0) -- +(-2,0);
\draw[dashed] (0,-0.2)+(0,0) -- +(60:1);
\draw[dashed] (0,-0.2)+(0,0) -- +(-60:2) -- +(-30:3.46);
\draw[dashed] (0,-0.2)+(-60:2) -- +(0.5,-2.60);
\end{tikzpicture}
\caption{Intersection of $\downY_i$ and $\downY_i+\epsilon$.}
\label{fig:tripod-intersection}
\end{figure}
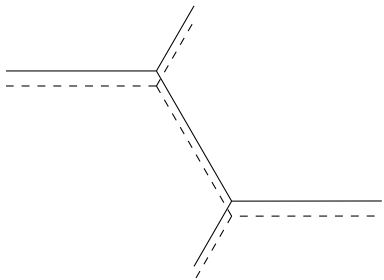
It follows that 
%\note{David: This is the key step and I don't see it.} 
$\downY_i \cap \downY_j = 2$ as claimed.

From this claim, it follows that $\downY = \cup_i \downY_i$ has $2\binom{n}{2} = n(n-1)$ self-intersections on the torus $\C / \Lambda(\downY)$.
The quotient map  $\C/ \Lambda(\downY) \to \C/ \Lambda$ to the original torus
has degree $n$,
so this implies that the tripod on $\C / \Lambda$ has $n-1$ self-intersections.
\end{proof}

\begin{Prop}\label{prop:regions}
Suppose $\downY$ is a tripod on the torus $\T = \C / \Lambda$
with only transverse self-intersections, with index $n$.
Then the complement $\T \setminus \downY$
consists of $n$ connected regions.
\end{Prop}
\begin{proof}
The tripod $\downY$ induces a cell structure on $\T$ as follows.
The vertices ($0$-cells) are the two tripod points of $\downY$ and all self-intersection points.
The edges ($1$-cells) are the segments of $\downY$ after subdividing along vertices,
and the faces ($2$-cells) are the connected components of $\T \setminus \downY$.
Note that the components of $\T \setminus \downY$ are simply connected,
because their lifts to the universal cover $\C$ are bounded.
%\note{(Clarify this?) David: I don't see why the "because" statement is the reason.}Let $c_0, c_1, c_2$ denote the number of vertices, edges, and faces respectively.
By Proposition~\ref{prop:intersections},
the number of vertices in this cell decomposition is $2 + (n-1) = n+1$.
To compute the number of edges,
we observe that two vertices have degree $3$ while the other $n-1$ vertices have degree $4$.
Thus 
$$
c_1 = \frac12 \sum_{v} \deg(v) = \frac12( 3\cdot 2 + 4 (n-1)) = 2n + 1 .
$$
Finally, we know that the Euler characteristic of $\T$ is zero,
so 
$$
c_2 = c_1 - c_0 = 2n + 1 - (n+1) = n,
$$
as claimed.
\end{proof}

\subsection{Counting tripods by spanning lattice}
Given a tripod $\downY$, recall that $\Lambda(\downY)$ denotes the minimal lattice which contains the endpoints of $\downY$;
we call $\Lambda(\downY)$ the {\em spanning lattice}  of $\downY$.

The association of $\downY$ to $\Lambda(\downY)$
defines a map
$$
(\text{tripods in $\Lambda$}) \to (\text{sublattices of }\Lambda).
$$
This map is surjective, but not injective.
%\note{Prove as proposition?}
For a fixed sublattice $\Lambda_0 \subset \Lambda$,
the set of tripods
$$
\{ \downY : \Lambda(\downY) = \Lambda_0\}
$$
is finite, but as $\Lambda_0$ varies the size of this preimage is unbounded.
In particular, the size of this preimage grows arbitrarily large asymptotically in proportion to the ratio
$$
(\text{length of second-shortest vector in }\Lambda_0 ) / (\text{length of shortest vector in }\Lambda_0) .
$$
%\note{Prove as proposition?}

\begin{Prop}
Given a lattice $\Lambda$ in $\C$,
there are finitely many tripods $\downY$ 
such that $\Lambda(\downY) = \Lambda$.
\end{Prop}
\begin{proof}
Without loss of generality, assume that $\Lambda$
has unit covolume.
Say a triangle with endpoints in $\Lambda$ is a {\em unit triangle} if it has area $1/2$.
By Theorem~\ref{thm:tripod-triangle},
it suffices to show that $\Lambda$ contains finitely many unit triangles with angles at most $2\pi /3$, up to translation so that one triangle vertex is at the origin.

Suppose $\Delta(0,z,w)$ is a unit triangle with angles at most $2\pi/3$.
If $\theta$ is the angle of the triangle at $0$,
then the triangle area satisfies
\begin{equation*}
\frac12 |z||w| \sin\theta = \mathrm{area}(\Delta(0,z,w)) = \frac12.
\end{equation*}
Up to translation, we may assume 
%one vertex of the triangle is at $0$ and 
that the largest angle of $\Delta(0,z,w)$ is at $0$,
so that $\theta \geq \pi/3$,
and our intial hypothesis is that $\theta \leq 2\pi/3$.
Therefore $\sin\theta \geq \sqrt{3}/2$,
which implies
\begin{equation*}
%\frac12 = \frac12 |z||w| \sin \theta \geq \frac{\sqrt 3}{4} |z||w|
%\qquad\Rightarrow\qquad
|z||w| = \frac1{\sin \theta} \leq \frac{2}{\sqrt 3}.
\end{equation*}
Let $L = \min \{ |z| : z\in \Lambda,\, z\neq 0\}$.
The above bound implies that $z$ and $w$ lie in the set $\{ z \in \Lambda : |z| \leq \frac{2}{ \sqrt 3 L}\}$
which is finite.
This verifies that there are finitely many unit triangles in $\Lambda$ up to translation with angles at most $2\pi/3$.
\end{proof}

\section{Lattice point counting}\label{sec:lattice}

\subsection{Lifting}
\label{sec:lift} 
We now describe how to turn our counting problem for tripods on the torus into a lattice point counting problem in $\C^2$. Given a tripod $\downY$ on $\T = \C/\Lambda$, we fix a lift to $\C$ by 
choosing the center point $\tilde p$ to lie in the sector
$0 \leq \arg(\tilde p) < 2\pi /3$.
%considering the triangle defined by the endpoints, and moving the largest angle of this triangle to $0$. 
The lifted tripod will have one endpoint at $0$.
Denote the other endpoints by $z, w \in \Lambda$, with $\arg(z) < \arg(w)$. 
(By $\arg(z) < \arg(w)$, we mean $\arg(z) < \arg(w) < \pi + \arg(z)$.)
See Figure~\ref{fig:tripod-ambiguity} for an example of determining a lift.
\begin{figure}[h]
\begin{tikzpicture}[scale = 2.0]
\draw(-1.5,0)--(1.5,0);
\draw(0, -1.5)--(0,1.5);
\node[right] at (1,1) {$w=i$};
%\draw(0,0)--(1,0)
\node[below] at (1,0) {$z=1$};
%\draw(1,0)--(0,1);
\node[above left] at (1-0.2,0.2) {$\tilde p$};

\draw[thick, blue](0,0)--(-0.21325, 0.21325);
\draw[thick, blue](1,0)--(1-0.21325, 0.21325);
\draw[thick, blue](0,-1)--(-0.21325, -1+0.21325);
\draw[thick, green](-1,0)--(-0.21325, 0.21325);
\draw[thick, green] (0,0) -- (1-0.21325, 0.21325);
\draw[thick, green](-1,-1)--(-0.21325, -1+0.21325);
\draw[thick, red] (0,1)--(-0.21325, 0.21325);
\draw[thick, red] (1,1)--(1-0.21325, 0.21325);
\draw[thick, red] (0,0) -- (-0.21325, -1+0.21325);

\end{tikzpicture}
\caption{Tripod $\downY$ lifted to $\C$, with endpoints $0,z,w$.
The lift with $\arg(\tilde p) \in [0,2\pi/3)$ is labelled.}
\label{fig:tripod-ambiguity}
\end{figure}
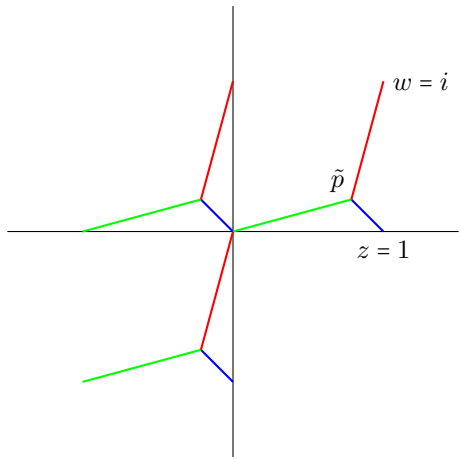

%\subsection{Largest angle}\label{sec:largest} For the largest angle of the triangle $\Delta(0, z, w)$ to be at $0$, it needs to be opposite the longest side, that is, 
%\begin{equation}\label{eq:largest} |z-w| > |z|, |w|.\end{equation}

\subsection{Angle bound}\label{sec:angle} 
By Theorem~\ref{thm:tripod-triangle}, a necessary and sufficient condition for $0,z,w$ to be endpoints of a tripod is that
\begin{equation}\label{eq:angle}
\text{the angles 
of the triangle $\Delta(0,z,w)$ are at most $2\pi/3$.}
\end{equation}

% note that the law of cosines yields 
%$$|z-w|^2 = |z|^2 + |w|^2 - 2|z||w| \cos \theta,$$ 
%and that cosine is decreasing on $[0, \pi)$, so 
%$$\theta \le 2\pi/3 
%\quad\iff\quad
% \cos \theta \geq -1/2 
%\quad\iff\quad 
%  -2 \cos \theta \leq 1.$$ 
%Rewriting, we obtain 
%\begin{equation}\label{eq:angle} 
%|z|^2 + |w|^2 + |z||w| \geq |z-w|^2 .
%\end{equation}
%We also need the angles of $\Delta(0,z,w)$ to be at most $2\pi/3$ at the vertices $z$ and $w$,
%which by the same argument is guaranteed by the ineqalities
%\begin{equation}
%\label{eq:angle-zw}
%\begin{gathered}
%|z|^2 + |w-z|^2 + |z| |w-z| \geq |w|^2 \\
%|w|^2 + |w-z|^2 + |w| |w-z| \geq |z|^2 .
%\end{gathered}
%\end{equation}

\subsection{Length bound}\label{sec:length} 
Finally, if we want $\ell(\downY(z, w)) \le R$, we need, by Lemma~\ref{lemma:length}, \begin{equation}\label{eq:length} |z e^{i\pi/3} + we^{-i\pi/3}| \le R\end{equation}

\noindent Putting \eqref{eq:angle} and \eqref{eq:length} together, we obtain:

\begin{lemma}\label{lemma:latticepoint} 
Suppose $\Lambda = \ZZ + \ZZ\tau$ is a lattice in $\CC \cong \RR^2$ with $\text{Im}(\tau) >0$.
The number of primitive tripods
$N_{\downY}(R,\Lambda)$ is given by the number of pairs $(z, w) \in \Lambda^2$ satisfying the following conditions:

\begin{equation}\label{eq:counting} 
\begin{gathered}
 z = a + b \tau, \quad w = c + d\tau,\quad \gcd(a,b,c,d) = 1, \\ 
 \arg(z) < \arg(w), \\ % < \pi + \arg(z), \\
% |z|^2 + |w|^2 + |z||w|  \geq |z-w|^2 , \\ 
% |z|^2 + |w-z|^2 + |z| |w-z| \geq |w|^2, \\
% |w|^2 + |w-z|^2 + |w| |w-z| \geq |z|^2,  \\
 \Delta(0,z,w) \text{ has all angles } \leq 2\pi/3, \\
 |z e^{i\pi/3} + we^{-i\pi/3}| \le R
 \text{ and } 0 \leq \arg(e^{i \pi/3}z + e^{-i \pi /3}w) < 2\pi /3 .
\end{gathered}
\end{equation}
\end{lemma}

\noindent 
The following corollary follows from standard lattice point counting results ~\cite[\S 24.10]{HardyWright}.

\begin{Cor}
\label{cor:latticepoint} 
%$$
%\lim_{R \rightarrow \infty} \frac{N_{\downY}(R,\Lambda)}{R^4} = \frac{1}{\zeta(4)} \frac{\vol(\Omega^{\downY})}{\covol(\Lambda)^2},
%$$ 
$$
\lim_{R \rightarrow \infty} \frac{N_{\downY}(R,\Lambda)}{R^4} = \frac{1}{\zeta(4)}\vol(\Omega^{\downY}),
$$ 
where 
$$
\Omega^\downY = \left\{ (z,w) \in \C^2 : 
\begin{gathered}
 \arg(z) < \arg(w) < \pi + \arg(z) , \\
% |z|^2 + |w|^2 + |z||w|  \geq |z-w|^2 , \\
% |z|^2 + |w-z|^2 + |z| |w-z| \geq |w|^2, \\
% |w|^2 + |w-z|^2 + |w| |w-z| \geq |z|^2, \\
 \Delta(0,z,w) \text{ has all angles } \leq 2\pi/3, \\
|e^{i \pi/3}z + e^{-i \pi /3}w| \leq 1 
\text{ and } 0 \leq \arg(e^{i \pi/3}z + e^{-i \pi /3}w) < 2\pi /3
\end{gathered} \right\}.
$$
%\begin{align}\label{eq:omega} 
%\Omega^{\downY} = \{ (z, w) \in \C^2 :\;  &|z|^2 + |w|^2 + |z||w| \geq |z- w|^2, \\
%\nonumber &  |z|^2 + |z-w|^2 + |z||z-w| \geq |w|^2 , \\ 
%\nonumber &  |w|^2 + |z-w|^2 + |w||z-w| \geq |z|^2 , \\ 
%\nonumber &  |z e^{i\pi/3} + w e^{-i\pi/3}| \le 1 ,
%\, 0 \leq \arg(e^{i \pi/3}z + e^{-i \pi /3}w) < 2\pi /3 \} . \nonumber 
%\end{align}
\end{Cor}

\begin{proof} $N_{\downY}(R,\Lambda)$ counts primitive points in $\Lambda^2=(\Z+\Z\tau)^2$ in the dilated set $R\Omega^{\downY}$. $\Omega^{\downY}$ is a compact region with smooth boundary (in $\C^2 \cong \R^4$), and so by~\cite[\S 24.10]{HardyWright}, $N_{\downY}(R,\Lambda)$ is asymptotic to 
%$$\frac{1}{\zeta(4)} \frac{\vol(\Omega^{\downY})}{\covol(\Lambda^2)} R^4,$$ 
$$\frac{1}{\zeta(4)} \vol(\Omega^{\downY}) R^4,$$ 
where the factor of $1/\zeta(4)$ is the probability that a random integer vector $(x, y, u, v) \in \Z^4$ is primitive, that is, $\gcd(x, y, u, v) = 1$.
\end{proof}

\section{Volumes}
\label{sec:volumes} 
\noindent To finish the proof of Theorem~\ref{theorem:main}, we need to compute the volume of $\Omega^{\downY}$. This is given by:

\begin{lemma}\label{lemma:volume} 
\begin{equation}\label{eq:volomega} 
\vol(\Omega^{\downY}) = \frac{\sqrt 3}{24} \pi 
= \frac14 \cdot \frac{2\pi}{3} \cdot \frac{\sqrt 3}{4}
\end{equation}
\end{lemma}

\begin{proof}[Proof of Theorem~\ref{theorem:main}]
Combining Corollary~\ref{cor:latticepoint} and Lemma~\ref{lemma:volume}, we obtain Theorem~\ref{theorem:main}.% \qed
\end{proof}

\subsection{Proof of Lemma~\ref{lemma:volume}} 
To prove Lemma~\ref{lemma:volume}, we will apply the following volume-preserving change of coordinates.
Let $\phi: \CC^2 \to \CC^2$
be defined by
$$
\phi \begin{pmatrix} z \\ w \end{pmatrix} 
= \begin{pmatrix} 1 & 0 \\
e^{i\pi /3} & e^{-i \pi/3} \end{pmatrix}
\begin{pmatrix} z \\ w \end{pmatrix} 
= \begin{pmatrix} z \\ e^{i \pi/3} z + e^{-i \pi /3} w \end{pmatrix} .
$$
The inverse map $\phi^{-1}$ with $u = e^{i \pi/3} z + e^{-i \pi /3} w$
is 
\begin{align*}
\phi^{-1} \begin{pmatrix} z \\ u \end{pmatrix} 
%&= (z, e^{i \pi/3} (u - e^{i \pi/3}z) ) \\
= \begin{pmatrix} 1 & 0 \\ e^{-i\pi/3} & e^{i\pi/3} \end{pmatrix}
\begin{pmatrix} z \\ u \end{pmatrix}
= \begin{pmatrix} z \\ e^{-i \pi/3}z + e^{i \pi/3} u \end{pmatrix}.
\end{align*}
If $\downY$ is a tripod with endpoints $0$, $z$, and $w$, 
and tripod point $p$,
then $u = e^{i \pi/3} z + e^{-i \pi /3} w$
is the third vertex of the equilateral triangle with 
vertices at $z$ and $w$.
We call $u$ the \emph{Toricelli point} of the tripod, (see Section~\ref{sec:fermat}).
By Lemma~\ref{lemma:length}, the point $u$ satisfies $\arg(u) = \arg(p)$
and $|u| = \ell(\downY) = |p| + |z-p| + |w-p|$.
% (see Lemma~\ref{lemma:length}).

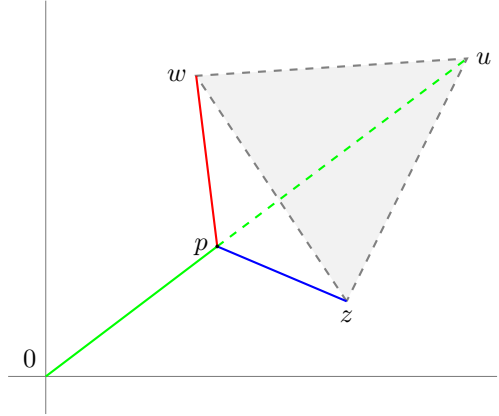
\begin{figure}[h]
\centering
\begin{tikzpicture}[scale=1.0]
\draw[gray] (-0.5,0) -- (6,0);
\draw[gray] (0,-0.5) -- (0,5);

\node[above left] at (0,0) {$0$};
\node[below] (z) at (4,1) {$z$};
\node[left] (w) at (2,4) {$w$};
\node[right] (u) at (5.60,4.23) {$u$};
\node[left] (p) at (2.28,1.73) {$p$};

\coordinate (z) at (4,1);
\coordinate (w) at (2,4);
\coordinate (p) at (2.28,1.73);
\coordinate (u) at (5.60,4.23);

\draw[thick, gray, dashed] (z) -- (w) -- (u) -- cycle;
\fill[gray,opacity=0.1] (z) -- (w) -- (u) -- cycle;

\draw[thick, blue] (p) -- (z);
\draw[thick, red] (p) -- (w);
\draw[thick, green] (0,0)--(p);
\draw[thick, green, dashed] (p) -- (u);
\filldraw (p) circle[radius=0.5pt];
\end{tikzpicture}
\caption{Tripod with endpoints $0,z,w$, and Toricelli point $u$.}
\end{figure}

Recall that $\Omega^\downY$ is defined as
$$
\Omega^\downY = \left\{ \twovector{z}{w} \in \CC^2 : \begin{gathered}
 \arg (z) < \arg(w) < \pi + \arg(z), \\
 \text{triangle } \Delta(0,z,w) \text{ has all angles }\leq 2\pi/3,  \\
|e^{i \pi/3}z + e^{-i \pi /3}w| \leq 1 
\text{ and } 0 \leq \arg(e^{i \pi/3}z + e^{-i \pi /3}w) < 2\pi /3
\end{gathered} \right\}
$$
Its image under $\phi$ is 
\begin{equation}
\phi(\Omega^\downY) = \left\{ \twovector{z}{u} \in \C^2 :
\begin{gathered}
% 0,\, z,\, \text{and}\, e^{-i \pi/3}z + e^{i \pi /3}u
% \text{ are endpoints of a tripod}, \\
\arg(z) < \arg(e^{-i \pi/3}z + e^{i \pi /3}u) < \pi + \arg(z) \\
\text{triangle } \Delta(0,z, e^{-i \pi/3}z + e^{i \pi /3}u) \text{ has all angles }\leq 2\pi/3, \\
|u| \leq 1, \quad 0 \leq \arg(u) < 2\pi/3 
\end{gathered}
\right\}.
\end{equation}

For $u\in \CC$ satisfying $|u|\leq 1$ and $0 \leq \arg(u) < 2\pi/3$, let
\begin{align*}
\phi(\Omega^\downY)_u &= \{ z \in \CC : (z,u) \in \phi(\Omega^\downY) \} \\
& = \left\{ z \in \CC : 
\begin{gathered}
0,\, z,\, \text{and}\, e^{-i \pi/3}z + e^{i \pi /3}u
\text{ are endpoints of a tripod} \\
\text{with }\arg(z) < \arg(e^{-i \pi/3}z + e^{i \pi /3}u) < \pi + \arg(z)
\end{gathered}
\right\}.
\end{align*}
and as a special case
\begin{equation}
\phi(\Omega^\downY)_1 = \left\{ z \in \C :
\begin{gathered}
0,\, z,\, \text{and}\, e^{-i \pi/3}z + e^{i \pi /3}
\text{ are endpoints of a tripod} \\
\text{with }\arg(z) < \arg(e^{-i \pi/3}z + e^{i \pi /3})
\end{gathered}
\right\}.
\end{equation}
%In other words, $\Omega^\downY_{00}$
%consists of all pairs $(z,w)$
%such that the tripod $\downY(z,w)$ has length $1$
%and whose center point $\tilde p(z,w)$ lies on the positive real axis.

\begin{lemma}
\label{lem:u-1-region}
The region $\phi(\Omega^\downY)_1$ 
is equal to the equilateral triangle 
$\Delta(0,1,e^{-i \pi/3}) \subset \CC$ (Fig.~\ref{fig:unit-triangle}).
\end{lemma}
\begin{proof}
By definition $\phi(\Omega^\downY)_1$ consists of points $z$ such that
the triangle
$\Delta(0, z, e^{-i\pi/3}z + e^{i\pi/3})$ 
has all angles at most $2\pi/3$.
By Theorem~\ref{thm:tripod-triangle},
such a triangle has an inscribed tripod.
Using Lemma~\ref{lemma:length} with $w = e^{-i\pi/3}z + e^{i\pi/3}$
and the computation
$$
e^{i\pi/3}z + e^{-i\pi/3}w
= e^{i\pi/3}z + (e^{-2i\pi/3}z + 1)
= 1,
$$
such a tripod has length $1$ and has its tripod point on the positive real axis.

Such a tripod has the following description, illustrated in Figure~\ref{fig:tripod-unit}.
%We may parametrize such a tripod as follows.
%Consider the tripod that consists of 
The center tripod point $p$ lies on the positive real axis, and let $a$ denote its distance from the origin.
%Consider a horizontal leg from the origin towards the positive real axis of length $a$,
The lower endpoint is $z$, and let $b$ denote its distance from $p$,
while the upper endpoint $w$ is at distance $1-a-b$ from $p$.
Each tripod leg has nonnegative length,
so we assume $0 \leq a,b \leq 1$ and $0 \leq a + b \leq 1$.
\begin{figure}[h]
\centering
\begin{tikzpicture}[scale=1.0]
\draw(-2,0) -- (2,0);
\draw(-1.5,-2) -- (-1.5,2);

\draw[thick, red](0,0)--(-1.5, 0)node[midway, above]{$a$};
\draw[thick, green](0,0)--(300:2)node[midway, right]{$b$};
\draw[thick, blue](0,0)--(60:1.6)node[midway, right]{$1-a-b$};

\node[right] at (300:2) {$z$};
\node[right] at (60:1.6) {$w$};
\node[below left] at (0,0) {$p$};
%\node at (1.5,1.5){$\C$};
\end{tikzpicture}
\caption{A tripod in $\C$.}
\label{fig:tripod-unit}
\end{figure}

\noindent In such a tripod,
$
z = a + be^{-i\pi/3}.
$
The points
$$
\{ z = a + b e^{-i \pi /3} :  a \geq 0,\, b\geq 0,\, a+b\leq 1 \}
$$
are exactly those inside the equilateral triangle with endpoints
$0$, $1$, and $e^{-i\pi/3}$.

%Note that $z$ must lie in the equilateral triangle 
%with vertices $0, 1, e^{-i \pi /3}$;
%see Figure~\ref{fig:unit-triangle}.
%Conversely, any $z$ in this triangle corresponds to some choice of tripod lengths
%$a,b,1-a-b$.
%If $z = s + it$, then $z$ corresponds to $a = -2t / \sqrt{3}$, $b = 1 - s + t/\sqrt{3}$.
%\begin{equation*}
%\begin{pmatrix} s \\ t \end{pmatrix}
%= \begin{pmatrix}
%1 - \frac12 a - b \\
%- \frac{\sqrt 3}{2} a
%\end{pmatrix}
%\qquad\Leftrightarrow\qquad
%\begin{pmatrix} a \\ b \end{pmatrix}
%= \begin{pmatrix}
%-\frac{2}{\sqrt 3} t \\
% 1 - s + \frac{1}{\sqrt 3} t
%\end{pmatrix}
%\end{equation*}

\begin{figure}[h]
\centering
\begin{tikzpicture}[scale=2.0]

\draw(-0.5,0) -- (1.5,0);
\draw(0,-1.2) -- (0,0.5);

\fill[gray!30] (0,0) -- (1,0) -- (300:1) -- cycle;

\draw[thick] (0,0) -- (1, 0)node[midway, above]{};
\draw[thick] (0,0) -- (300:1)node[midway, right]{};
\draw[thick] (1,0) -- (300:1)node[midway, right]{};

%\node at (1.5,1.5){$\C$};
\end{tikzpicture}
\caption{The unit equilateral triangle $\Delta(0,1,e^{- i \pi /3})$ in $\C$.}
\label{fig:unit-triangle}
\end{figure}
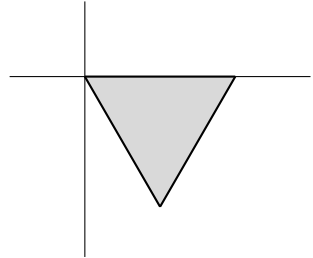

It follows that
$
\phi(\Omega^\downY)_1 = \Delta(0,1,e^{-i\pi/3})
$
as claimed.
\end{proof}

%By a change of coordinates,
%the parametrization
%$$
%(a,b) \mapsto 
%\begin{pmatrix} x \\ y \\ u \\ v \end{pmatrix}
%=
%\begin{pmatrix}
% (1-b) + \cos(-2\pi/3) a \\
% \sin(-2\pi/3) a \\
% (1 - a) + \cos(2\pi/3)b \\
% \sin(2\pi/3) b
%\end{pmatrix}
%$$
%is equivalent to 
%$$
%(s,t) \mapsto 
%\begin{pmatrix} x \\ y \\ u \\ v \end{pmatrix}
%=
%\begin{pmatrix}
% s \\
% t \\
% 1 - \cos(\frac{\pi}{3}) (1-s) + \sin(\frac{\pi}{3}) t \\
% \sin(\frac{\pi}{3}) (1-s) + \cos(\frac{\pi}{3}) t
%\end{pmatrix}
%=
%\begin{pmatrix}
% s \\
% t \\
% \cos(\frac{\pi}{3}) + \cos(\frac{\pi}{3}) s + \sin(\frac{\pi}{3}) t \\[0.5em]
% \sin(\frac{\pi}{3}) - \sin(\frac{\pi}{3}) s + \cos(\frac{\pi}{3}) t 
%\end{pmatrix}
%$$

\begin{lemma}
\label{lem:u-region}
The region $\phi(\Omega^\downY)_u$ 
has area $\frac{\sqrt{3}}{4}|u|^2$.
\end{lemma}
\begin{proof}
When $u = 1$, the claim follows from Lemma~\ref{lem:u-1-region}.
For arbitrary $u$, the association $u \mapsto \phi(\Omega^\downY)_u$ is equivariant under multiplication,
i.e.
$$
\phi(\Omega^\downY)_u
= \{ uz : z \in \phi(\Omega^\downY)_1\} .
$$
The region $\phi(\Omega^\downY)_u$ has real dimension $2$,
so it follows that
$$
\vol(\phi(\Omega^\downY)_u) = |u|^2 \vol(\phi(\Omega^\downY)_1) 
= \frac{\sqrt{3}}{4} |u|^2
$$
as claimed.
\end{proof}

\begin{proof}[Proof of Lemma~\ref{lemma:volume}]
First, note that 
$\vol(\Omega^\downY) = \vol(\phi(\Omega^\downY))$
since $\phi$ has unit-magnitude Jacobian:
$$
|Jac(\phi)| = 
\left| \det \begin{pmatrix}
1 & 0 \\
e^{i\pi /3} & e^{-i\pi/3}
\end{pmatrix} \right|
= 1 .
$$
Then to compute $\phi(\Omega^\downY)$:
we slice the region according to the $u$-coordinate.
$$
\vol(\phi(\Omega^\downY)) = \int\int_{(z,u) \in \phi(\Omega^\downY)} dz\,du
= \int_{u} \vol(\phi(\Omega^\downY)_u) du.
$$

From the definition of $\phi(\Omega^\downY)$,
it is clear that the region $\phi(\Omega^\downY)_u$ is nonempty only if
$$
u = r e^{i \theta} 
\quad \text{where } 0 \leq r \leq 1 \text{ and } 0 \leq \theta \leq 2\pi/3 .
$$
Using this substitution $u = re^{i\theta}$ and 
the identity $\vol(\phi(\Omega^\downY)_u) = \frac{\sqrt{3}}{4} |u|^2 = \frac{\sqrt{3}}{4} r^2$ from Lemma~\ref{lem:u-region},
we have
\begin{align*}
\vol(\phi(\Omega^\downY)) = \int_{\theta=0}^{2\pi/3} \int_{r=0}^1 \frac{\sqrt 3}{4} r^2 \cdot r\,dr\, d\theta
= \frac{\sqrt 3}{4} \int_{\theta=0}^{2\pi/3}  d\theta\int_{r=0}^1 r^3 \,dr
= \frac{\sqrt 3}{4} \cdot \frac{2\pi}{3} \cdot \frac1{4}
= \frac{\sqrt 3 \pi}{24}
\end{align*}
as claimed.
\end{proof}

\section{Non-reduced Tripods}\label{sec:NonRedTripods}

\begin{proof}[Proof of Thm.~\ref{theorem:AEreduced}]
We will consider a lattice with a non-reduced tripod and show that the real and imaginary parts of the complex number determining the lattice are related by an equation.  From this, we conclude that a lattice admitting a non-reduced tripod must lie in a countable union of positive codimension subset of the space of lattices.

Consider a lattice $\Lambda$ in $\bC$.  Consider a tripod in this lattice with the property that at least one of its legs has a lattice point on its interior.  We consider the following transformations.  Let $\ell_1$ be a leg that has a lattice point in its interior.  Without loss of generality, translate the lattice so that the endpoint of $\ell_1$ is the origin.  Next observe that rotating a tripod preserves the tripod property.  Therefore, we rotate the lattice $\Lambda$ about the origin to a new lattice such that $\ell_1$ lies in the positive real axis.  Next we scale the entire lattice so that one of its basis vectors is $1$ and we write $\Lambda' = \bZ \oplus \tau \bZ$, where $\tau = s + it$.  We make no claims that $\Lambda'$ has unit covolume.  From now on we work entirely with respect to this tripod in the lattice $\Lambda'$.

Let $z$ and $w$ be the other endpoints of the tripod.  Then there exist integers $a_z, b_z, a_w, b_w \in \bZ$ such that $z = a_z + b_z \tau$ and $w = a_w + b_w \tau$.  Furthermore, the tripod point is simply given by a real number $r_p$ and the assumption that $\ell_1$ has a lattice point on it implies $r_p \geq 1$.  The angle conditions on the tripod points imply that
$$\left(\frac{1}{2} + i \frac{\sqrt{3}}{2}\right) \left(a_z + b_z \tau - r_p \right) \in \bR$$
and
$$\left(\frac{1}{2} - i \frac{\sqrt{3}}{2}\right) \left(a_w + b_w \tau - r_p \right) \in \bR.$$
In particular, the imaginary parts of both quantities are $0$.  After multiplying by $2$, this yields the equations
$$b_z t + \sqrt{3}\left(a_z + b_z s - r_p \right) = 0$$
$$b_w t - \sqrt{3}\left(a_w + b_w s - r_p \right) = 0.$$
It is convenient to define $t = t' \sqrt{3}$ to get
\begin{eqnarray}
\label{ImPartsEqn}
b_z t' + a_z + b_z s - r_p = 0 \\
b_w t' - a_w - b_w s + r_p  = 0.
\end{eqnarray}
Adding these equations yields
$$(b_z + b_w)t' + (a_z - a_w) + (b_z - b_w)s = 0.$$

First we consider the case where $b_z + b_w \not= 0$.  In this case, $t' = q_1 + q_2 s$ for some $q_1, q_2 \in \Q$.  However, this implies that $\tau$, which defines the lattice $\Lambda$ is not free to be any point in the complex plane because its real and complex parts are related by an equation.  This completes the proof in this case.

Next, we consider the case where $b_z + b_w = 0$.  We claim that $b_z - b_w \not= 0$ because otherwise, we would have $b_z = b_w = 0$ and this would contradict the assumption that $0$, $z$ and $w$ form a tripod in $\Lambda$.  Therefore, we conclude that $s \in \Q$ in which case we have again reduced to a measure zero subset of the space of lattices.
\end{proof}

\subsection{Many nonreduced tripods}

In this section fix $\Lambda = \ZZ + \ZZ\zeta$ as the triangular lattice,
where $\zeta = e^{2\pi i/6}$ denotes a sixth root of unity.

\begin{proof}[Proof of Thm.~\ref{theorem:MaximalReduced}]
Consider a tripod $\downY$ in $\Lambda$
with endpoints at $0$, $z = a + b\zeta$, and $w = c + d\zeta$,
where $a,b,c,d$ are integers.
Assume that $\arg(z) < \arg(w)$.
Let $u$ denote the other vertex  of the equilateral triangle with vertices at $z$ and $w$.
the vertex $u$ satisfies
\begin{align*}
u &= \zeta z + \zeta^{-1} w 
= c\zeta^{-1} + d + a\zeta + b\zeta^2
= c(1 - \zeta) + d + a\zeta + b(\zeta - 1) \\
&= (-b + c + d) + (a + b - c)\zeta.
\end{align*}
This shows that $u \in \Lambda$.
The tripod point $p$ of $\downY$ is a positive real multiple of $u$.

Conversely, if $z, w, u$ are related as above and $z,u \in \Lambda$,
then it is straightforward to check that also $w\in \Lambda$.

Recall that a tripod $\downY$ is {\em nonreduced}
if its interior contains a lattice point.
If the leg $0p$ contains a lattice point in its interior,
then $0u \supset 0p$ also contains a lattice point so $u$ must be a nonprimitive
lattice point of $\Lambda$.
Conversely if $u$ is a nonprimitive lattice point so that $0u$ contains lattice points in its interior,
a sufficient condition for $0p$ to contain a lattice point in its interior is that
$$
\ell(0p) > \frac12 \ell(0u).
$$

Therefore as a lower bound for the number of nonreduced lattices,
we have
$$
\# \{ \text{nonreduced } \downY \text{ in }\Lambda\}
\geq  \sum_{\substack{\text{nonprimitive}\\ u \in \Lambda}}
\# \left\{ \text{tripods } \downY : 
\begin{gathered}
\text{tripod point } p \in 0u, \,
\ell(0p) > \frac12 \ell(0u), \\
0p\text{ is longest  leg of }\downY
\end{gathered}\right\}.
$$
The condition that $0p$ is the longest tripod leg is needed to avoid overcounting
on the torus $\RR^2 / \Lambda$.
Note that the length of the tripod $\ell(\downY)$ is equal to the distance $|u| = \ell(0u)$.

Suppose we fix some $u$ with $|u| = R$.
Then the set
$$
\{ \text{tripods } \downY : 
\begin{gathered}
\text{tripod point } p \in 0u
\end{gathered}\}
$$
is parametrized by choosing $w$ inside the equilateral triangle 
$\Delta(0,u,\zeta u)$ with side length $R$,
while the set
$$
\left\{ \text{tripods } \downY : 
\begin{gathered}
\text{tripod point } p \in 0u, \, 
0p \text{ is longest leg of }\downY
\end{gathered}\right\}
$$
is parametrized by choosing $w$ inside a subregion of the previous triangle
of one-third size, see Figure~\ref{fig:triangle-subset} (middle).
Finally the set
$$
\left\{ \text{tripods } \downY : 
\begin{gathered}
\text{tripod point } p \in 0u, \, 
0p \text{ is longest leg of }\downY, \\
\ell(0p) > \frac12 \ell(0u)
\end{gathered}\right\}
$$
is parametrized by choosing $w$ inside the subregion shown in 
Figure~\ref{fig:triangle-subset} (right), which has one-fourth the size of the original triangle.

Therefore for fixed lattice point $u$,
the number of tripods
\begin{align*}
\# \left\{ \text{tripods } \downY : 
\begin{gathered}
\text{tripod point } p \in 0u, \,
\ell(0p) > \frac12 \ell(0u), \\
0p\text{ is longest  leg of }\downY
\end{gathered}\right\}
&= \#\{ \text{lattice points in subregion of } \Delta(0,u, \zeta u) \} \\
&= \frac{\textrm{vol}(\text{triangle subregion})}{\textrm{covol}(\Lambda)}
+ O(R) \\
&= \frac{\sqrt{3} R^2 / 16}{\sqrt{3}/2} + O(R)\\
&= \frac{1}{8}R^2 + O(R).
\end{align*}
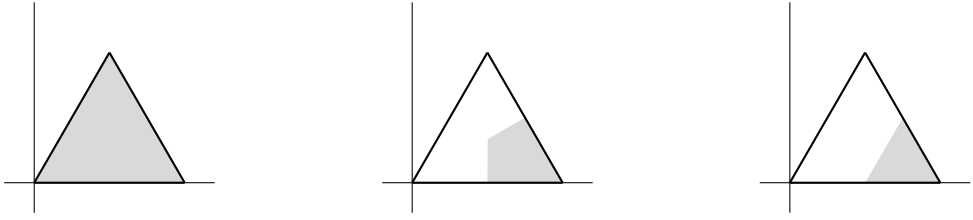
\begin{figure}[h]
\centering
\begin{minipage}[c]{0.3\textwidth}
% whole triangle
\begin{tikzpicture}[scale=2.0]

\draw(-0.2,0) -- (1.2,0);
\draw(0,-0.2) -- (0,1.2);

\fill[gray!30] (0,0) -- (1,0) -- (60:1) -- cycle;

\draw[thick] (0,0) -- (1, 0)node[midway, above]{};
\draw[thick] (0,0) -- (60:1)node[midway, right]{};
\draw[thick] (1,0) -- (60:1)node[midway, right]{};

%\node at (1.5,1.5){$\C$};
\end{tikzpicture}
\end{minipage}
\begin{minipage}[c]{0.3\textwidth}
% one-third triangle
\begin{tikzpicture}[scale=2.0]

\draw(-0.2,0) -- (1.2,0);
\draw(0,-0.2) -- (0,1.2);

\fill[gray!30] (0.5,0) -- (1,0) -- (30:0.87) -- (30:0.58) -- cycle;

\draw[thick] (0,0) -- (1, 0)node[midway, above]{};
\draw[thick] (0,0) -- (60:1)node[midway, right]{};
\draw[thick] (1,0) -- (60:1)node[midway, right]{};

%\node at (1.5,1.5){$\C$};
\end{tikzpicture}
\end{minipage}
\begin{minipage}[c]{0.3\textwidth}
% one-fourth triangle
\begin{tikzpicture}[scale=2.0]

\draw(-0.2,0) -- (1.2,0);
\draw(0,-0.2) -- (0,1.2);

\fill[gray!30] (0.5,0) -- (1,0) -- (30:0.87) -- cycle;

\draw[thick] (0,0) -- (1, 0)node[midway, above]{};
\draw[thick] (0,0) -- (60:1)node[midway, right]{};
\draw[thick] (1,0) -- (60:1)node[midway, right]{};

%\node at (1.5,1.5){$\C$};
\end{tikzpicture}
\end{minipage}
\caption{An equilateral triangle and relevant subregions.}
\label{fig:triangle-subset}
\end{figure}

Now it remains to sum over the possible choices of $u$.
Summing over all lattice points $u$,
we would have
\begin{align*}
\sum_{\substack{ u \in \Lambda \\ |u| \leq R}}\# \left\{ \text{tripods } \downY : 
\begin{gathered}
\text{tripod point } p \in 0u, \,
\ell(0p) > \frac12 \ell(0u), \\
0p\text{ is longest  leg of }\downY
\end{gathered} \right\}
&= 
\sum_{\substack{ u \in \Lambda \\ |u|  = r \leq R}} \frac18 r^2 + O(r) \\
&= \int_{r=0}^R \left( \frac{\pi}{4} r^3 + O(r^2) \right) dr \\
&= \frac{\pi}{16}R^4 + O(R^3).
\end{align*}
If we instead sum over only those $u$ which are nonprimitive lattice points,
our answer changes asymptotically by a constant factor
$$
1 - \zeta(2)^{-1} 
= \lim_{R \to \infty} \frac{\#\{u \in \Lambda : \, |u|\leq R,\, u \text{ nonprimitive}\}}{\#\{u \in \Lambda,\, |u|\leq R\}}
= 1 - \frac{6}{\pi^2} \approx 0.392.
$$
Therefore,
\begin{align*}
\# \{ \text{nonreduced } \downY \text{ in }\Lambda : \ell(\downY)\leq R \}
= \left(1 - \frac{6}{\pi^2}\right) \frac{\pi}{16} R^4 + O(R^3),
\end{align*}
so we may take any postive constant $C < (1 - \frac{6}{\pi^2}) \frac{\pi}{16} \approx 0.0770$.
\end{proof}

Note that the total number of tripods in $\Lambda$ satisfies
$$
\#\{ \downY \text{ in }\Lambda : \, \ell(\downY) \leq R\} \sim \frac{\pi}{12}R^4 .
$$
\begin{Cor}
The number of nonreduced, primitive tripods in $\Lambda$ satisfies
$$
\# \{ \downY \text{ nonreduced and primitive} : \ell(\downY) \leq R \} \geq C R^4
$$
for some positive constant $C$, for sufficiently large $R$.
\end{Cor}
\begin{proof}
The previous theorem showed that
\begin{align*}
\#\{ \text{nonreduced }\downY \text{ in }\Lambda : \, \ell(\downY) \leq R \} 
&\gtrsim \left( 1 - \frac{6}{\pi^2}\right) \frac{\pi}{16} R^4 \\
&\sim \left(1 -\frac{6}{\pi^2}\right) \frac{3}{4} \#\{\text{total }\downY :\, \ell(\downY)\leq R\},
\end{align*}
where the constant
$$
C_1 = \left(1 -\frac{6}{\pi^2}\right) \frac{3}{4} \approx 0.294 .
$$
We also know that the number of primitive tripods satisfies
\begin{align*}
\#\{ \text{primitive }\downY :\, \ell(\downY) \leq R\}
&\sim \zeta(4)^{-1} \#\{ \text{total }\downY :\, \ell(\downY) \leq R\}, 
%\\
%&\sim \frac{90}{\pi^4} \#\{ \text{total }\downY,\, \ell(\downY) \leq R\}
\end{align*}
where
$$
C_2 = \zeta(4)^{-1} \approx 0.924 .
$$
Together these bounds imply that
\begin{align*}
\#\{ \text{nonreduced, primitive }\downY :\, \ell(\downY)\leq R\} 
&\gtrsim (C_1 + C_2 - 1) \#\{ \text{total }\downY,\,\ell(\downY)\leq R\} \\
&\sim (C_1 + C_2 - 1) \frac{\pi}{12} R^4 .
\end{align*}
We may take $C$ to be any positive constant less than $(C_1 + C_2 - 1) \frac{\pi}{12}$.
\end{proof}

\appendix
\section{Numerics}\label{sec:numerics} %\subsection{Sage code} 

We obtained experimental evidence for Theorem~\ref{theorem:main} using the following elementary Sage code, which computes the number of tripods in $\Z[i]$ of length at most $R$.

%\begin{Verbatim}
%from itertools import product
%from time import time
%def is_prim(a,b,c,d): 
%	return gcd(gcd(a,b), gcd(c,d))==1 
%#checks that a+bi and c+di are primitive
%def is_longest(a,b,c,d): 
%	return min(a^2+b^2, c^2+d^2) > 2*a*c+2*b*d
%# checks that the largest angle of the triangle with vertices at 0, a+bi, c+di
%# is at 0, by checking that the length of the side opposite 0, 
%# whose length is |(a-c) + (b-d)i|, is longest.
%def is_positively_oriented(a, b, c, d):
%	return a*d-b*c > 0
%# checks that the triangle with vertices at 0, a+bi, c+di is positively oriented
%def is_tripod(a,b,c,d):
%	return is_longest(a, b, c, d) and RR(2*a*c+2*b*d+ sqrt((a^2+b^2)*(c^2+d^2)))>0
%# checks that the triangle with vertices at 0, a+bi, c+di admits a tripod
%def tripod_length_squared(a,b,c, d, R):
%	return RR((a-c)^2 + (b-d)^2 + a*c + b*d + sqrt(3)*(a*d-b*c)) < RR(R^2)
%#computes the length-squared of the tripod	
%def tripod_counts(R): #guess is 0.20947986097*R^4 = RR(15*sqrt(3)/(4*pi^3)) R^4
%	L=[(a,b, c, d) for (a,b,c,d) in product(range(-1.5*R, 1.5*R), repeat=4) 
%	if is_prim(a,b,c,d) and is_longest(a,b,c,d) and is_positively_oriented(a,b,c,d)]
%	return len(L)
%# returns the length of the list of tripods of length at most $R$
%\end{Verbatim}

\begin{Verbatim}
from itertools import product
from time import time

def is_prim(a,b,c,d):
    '''checks that a+bi and c+di are primitive'''
    return gcd(gcd(a,b), gcd(c,d))==1 
def is_longest(a,b,c,d):
    '''checks that the largest angle of the triangle with vertices at 0, a+bi, c+di
    is at 0, by checking that the length of the side opposite 0,
    whose length is |(a-c) + (b-d)i|, is longest.'''
    return min(a^2+b^2, c^2+d^2) > 2*a*c+2*b*d
def is_positively_oriented(a, b, c, d):
    '''checks that the triangle with vertices at 0, a+bi, c+di is positively oriented'''
    return a*d-b*c > 0
def is_tripod(a,b,c,d):
    '''checks that the triangle with vertices at 0, a+bi, c+di admits a tripod'''
    return is_longest(a, b, c, d) and RR(2*a*c+2*b*d+ sqrt((a^2+b^2)*(c^2+d^2)))>0
def tripod_length_squared(a,b,c, d, R):
    '''computes the length-squared of the tripod'''
    return RR((a-c)^2 + (b-d)^2 + a*c + b*d + sqrt(3)*(a*d-b*c)) < RR(R^2)
def tripod_counts(R): #guess is 0.20947986097*R^4 = RR(15*sqrt(3)/(4*pi^3)) R^4
    '''returns the length of the list of tripods of length at most $R$'''
    L=[(a,b, c, d) for (a,b,c,d) in product(range(-1.5*R, 1.5*R), repeat=4)
       if is_prim(a,b,c,d) and is_longest(a,b,c,d) and is_positively_oriented(a,b,c,d)]
    return len(L)
\end{Verbatim}

\noindent For $R=35,$ this yields

\begin{Verbatim}
tripod_counts(35) = 312488
	\end{Verbatim}

\noindent For comparison, $$\left|\frac{312488}{(35)^4} - \frac{15\sqrt{3}}{4\pi^3}\right| = 0.00124129370635984\ldots.$$ We make no claims that this Sage code is particularly efficient.
\bibliographystyle{alpha}
\bibliography{TripodsTorus}{}

\end{document}